\documentclass[reqno,11pt]{article}

\usepackage{a4wide,bm}
\usepackage[pdfborder={0 0 0}]{hyperref}  
\usepackage{color,eucal,enumerate,mathrsfs}
\usepackage[normalem]{ulem}
\usepackage{amsmath,amssymb,epsfig,bbm}
\numberwithin{equation}{section}
\usepackage{dsfont}
\usepackage{tikz}

\usepackage[latin1]{inputenc}

\renewcommand{\H}{\mathbb{H}}

\newcommand{\N}{\mathbb{N}}

\newcommand{\R}{\mathbb{R}}

\newcommand{\mm}{{\mbox{\boldmath$m$}}}

\newcommand{\sfd}{{\sf d}}

\newcommand{\Kliminf}{K\kern-3pt-\kern-2pt\mathop{\rm lim\,inf}\limits}  
\newcommand{\supp}{\mathop{\rm supp}\nolimits}   
\renewcommand{\d}{{\mathrm d}}

\newcommand{\restr}[1]{\lower3pt\hbox{$|_{#1}$}}
\newcommand{\la}{{\langle}}                  
\newcommand{\ra}{{\rangle}}
  
\newcommand{\nchi}{{\raise.3ex\hbox{$\chi$}}}

\newcommand{\fr}{\penalty-20\null\hfill$\blacksquare$}                      

\renewcommand{\mm}{\mathfrak m}                                

\newenvironment{proof}{\removelastskip\par\medskip   
\noindent{\em proof} \rm}{\penalty-20\null\hfill$\square$\par\medbreak}

\newtheorem{theorem}{Theorem}[section]

\newtheorem{lemma}[theorem]{Lemma}
\newtheorem{proposition}[theorem]{Proposition}

\newtheorem{conjecture}[theorem]{Conjecture}

\newtheorem{definition}[theorem]{Definition}

\newtheorem{remark}[theorem]{Remark}

\newcommand{\test}[1]{{\rm Test}(#1)}

\newcommand{\X}{{\rm M}}

\newcommand{\vol}{{\rm vol}}

\renewcommand{\div}{{\rm div}}
\newcommand{\vsm}{{\rm TestV}(\X)}
\newcommand{\fsm}{{\rm Test}(\X)}

\newcommand{\HS}{{\lower.3ex\hbox{\scriptsize{\sf HS}}}}
\renewcommand{\H}[1]{{\rm Hess}(#1)}

\newcommand{\RCD}{{\sf RCD}}
\newcommand{\CD}{{\sf CD}}

\newcommand{\ru}[3]{{\bold R}(#1,#2)(#3)}
\newcommand{\rd}[4]{\bm{\mathcal R}(#1,#2,#3,#4)}

\newcommand{\bsq}[2]{\pmb [#1,#2\pmb]}
\newcommand{\bnabla}{\bm\nabla}
\newcommand{\distrlie}[2]{\bsq{#1}{#2}}

\title{Riemann curvature tensor on ${\sf RCD}$ spaces and possible applications}

\begin{document}

\author{Nicola Gigli \thanks{SISSA, Trieste. email: \textsf{ngigli@sissa.it}}
   }

\maketitle

\begin{abstract}  
We show that on every $\RCD$ spaces it is possible to introduce, by a distributional-like approach, a Riemann curvature tensor. 

Since after the works of Petrunin and Zhang-Zhu we know that finite dimensional Alexandrov spaces are $\RCD$ spaces, our construction applies in particular to the Alexandrov setting.  We conjecture that an $\RCD$ space is Alexandrov if and only if the sectional curvature - defined in terms of such abstract Riemann tensor - is bounded from below.

\bigskip

{\centerline{\bf R\'esum\'e}}

\medskip

Nous montrons que sur chaque espace $\RCD$ il est possible d'introduire, par une approche distributionnelle, un tenseur de courbure de Riemann.

Puisque apr\`es les travaux de Petrunin et de Zhang-Zhu nous savons que les espaces d'Alexandrov de dimension finie sont des espaces  RCD, notre construction s'applique en particulier au cadre d'Alexandrov. Nous conjecturons qu'un espace $\RCD$ est Alexandrov si et seulement si la courbure sectionnelle - d\'efinie en termes de ce tenseur de Riemann abstrait - est born\'ee par  dessous.

\end{abstract}

\tableofcontents

\section{Introduction} 

One of the outcomes of the tensor calculus on $\RCD$ spaces built in \cite{Gigli14} is the existence of a measure-valued Ricci tensor. An example of application of this object to the study of the geometry of such spaces is given by the paper \cite{Han2018} where, inspired by some more formal computations due to Sturm \cite{Sturm14}, it is shown that transformations of the metric-measure structure (like, e.g., conformal ones) alter lower Ricci curvature bounds as in the smooth context. 

Stability of the ${\sf CAT}(0)$ condition under appropriate conformal transformations has been shown in \cite{LytSta18} with different techniques, in particular without relying on any sort of tensor calculus, but, notably, the same kind of question in the context of Alexandrov spaces is open, see e.g.\ \cite{Petrunin07}.  In this paper we want to propose an attack plan to such problem: our idea is to leverage on the available calculus tools on the $\RCD$ setting to produce a sectional curvature tensor on Alexandrov spaces, so that then, hopefully, appropriate transformation formulas can be studied. In this direction, recall that thanks to Petrunin's \cite{Petrunin11} and Zhang-Zhu's results \cite{ZhangZhu10} we know that a $n$-dimensional Alexandrov space with curvature bounded from below by $k$ is $\CD(k(n-1),n)$, in line with the smooth setting. Since moreover Alexandrov spaces are infinitesimally Hilbertian (see \cite{KMS01}; the notion of Sobolev function used in this work is different from the one - introduced in \cite{Cheeger00} - typically adopted on metric measure spaces, yet the two are easily seen to be equivalent thanks to the structural properties of Alexandrov spaces pointed out in  \cite{KMS01}), we see that a finite dimensional Alexandrov space is always a $\RCD$ space and thus all the calculus tools available in the latter setting are at disposal also in the former.

An interesting fact, and the main point of this current manuscript, is that on a $\RCD$ space it is possibile to give a meaning, in a kind of distributional sense, to the full Riemann curvature tensor. To see why, start recalling that the typical term in the definition of $R(X,Y,Z,W)$ on the smooth setting is $\la\nabla_X\nabla_YZ,W\ra$ and observe that it seems hard to give a `direct' meaning to such expression on $\RCD$ spaces, because it is unclear whether there are vector fields regular enough to be covariantly differentiated twice (we believe in general there aren't many of these). Instead, we have `many' (i.e.\ $L^2$-dense) vector fields which are bounded and with covariant derivative in $L^2$, and this allows to give a `weak' meaning to such object: indeed, multiplying  $\la\nabla_X\nabla_YZ,W\ra$  by a smooth function $f$, integrating w.r.t.\ the volume measure and then integrating by parts we obtain
\[
\int f \la\nabla_X\nabla_YZ,W\ra\,\d\vol=-\int\la\nabla_YZ,\nabla_X(fW)\ra+f\div X\la\nabla_YZ,W\ra\,\d\vol,
\]
and from what we just said we see that the right hand side is well-defined for `many' vector fields and functions. Thus following this line of thought it is possible to define, for any $X,Y,Z,W$ sufficiently smooth vector fields, the Riemann curvature tensor $\rd XYZW$ as a real valued operator acting on a space of sufficiently smooth functions. In more precise terms, we shall work with the spaces of \emph{test functions} $\test\X$ and \emph{test vector fields} $\vsm$, see the beginning of the next section for the definitions.

\bigskip

In particular we also have a sectional curvature operator on $\RCD$ spaces, and a fortiori  on finite dimensional Alexandrov spaces. We conjecture that lower bounds of such sectional curvature  are equivalent to the Alexandrov condition in the following sense:
\begin{conjecture} Let $(\X,\sfd)$ be a complete and separable metric space. Then the following are equivalent:
\begin{itemize}
\item[i)] $(\X,\sfd)$ is a $n$-dimensional Alexandrov space of curvature bounded from below by $k\in\R$.
\item[ii)] $(\X,\sfd,\mathcal H^n)$ is a $\RCD(k(n-1),n)$ space, $\supp(\mathcal H^n)=\X$ and 
\[
\rd XYYX(f)\geq k\int f|X\wedge Y|^2\,\d\mm\qquad\forall f\in \test\X,\ f\geq0,\ X,Y\in\vsm.
\]
\end{itemize}
\end{conjecture}


Here and below $\mathcal H^n$ is the $n$-dimensional Hausdorff measure. Let us collect some comments about this conjecture. First of all, the already recalled results by Kuwae-Machigashira-Shioya, Petrunin and Zhang-Zhu grant that if $(i)$ holds then $(\X,\sfd,\mathcal H^n)$ is a $\RCD(k(n-1),n)$ space such that $\supp(\mathcal H^n)=\X$. This is all is known so far about the relation between $(i)$ and $(ii)$. We also point out that $\RCD(K,N)$ spaces for which the reference measure is $\mathcal H^N$ are called \emph{non-collapsed} $\RCD$ spaces (see \cite{GDP17}) and are more regular than generic $\RCD$ spaces: they are the synthetic analogue of the \emph{non-collapsed Ricci limit spaces} introduced by Cheeger-Colding in \cite{Cheeger-Colding97I}.

Finally, we remark that the abstract tensor calculus developed in \cite{Gigli14}  has already been shown to have strong links with the geometry of $\RCD$ spaces. As non-exhaustive list of recent results where it has been used as key (but certainly not exclusive) tool, let us mention: the link between dimension of the first cohomology group and geometry of the underlying space \cite{GR17}, the constant dimension property of $\RCD$ spaces \cite{BS18}, the regularity that comes from imposing both a lower Ricci bound and an upper sectional bound on the space \cite{KK18}.

\section{Riemann curvature tensor on $\RCD(K,\infty)$ spaces}

To keep this note short, we shall assume the reader familiar with the notion of $\RCD$ space and with the calculus developed in    \cite{Gigli14}. Throughout this note  $(\X,\sfd,\mm)$ will be a fixed $\RCD(K,\infty)$ space, $K\in\R$.

We recall that in \cite{Savare13} it has been introduced the space of `test functions'
\[
\test\X:=\big\{f\in D(\Delta)\subset W^{1,2}(\X)\ :\ {\rm Lip}(f)<\infty,\ f\in L^\infty(\X),\ \Delta f\in W^{1,2}(\X)\big\}
\]
and proved that this is an algebra dense in $W^{1,2}(\X)$. Then in \cite{Gigli14} the space of `test vector fields' has been defined as
\[
\vsm:=\Big\{\sum_{i=1}^nf_i\nabla g_i\ :\ i\in \N,\ f_i,g_i\in \test\X\Big\}.
\]
We won't put any topology on the vector space $\vsm$ (see Remark \ref{rem:distr} for comments in this direction); here we just notice that the product of a function in $\fsm$ and a vector field in $\vsm$ is still an element of $\vsm$, i.e.\ $\vsm$ is a module over $\fsm$.
\begin{definition}[The space $\vsm'$] The vector space $\vsm'$ is the space of all linear maps from  $\vsm$ to $\R$, i.e.\ the dual of $\vsm$ in the algebraic sense.
\end{definition}
The vector space $\vsm'$  comes with the structure of module over $\fsm$, the product of  an operator  $T\in \vsm'$ and a function  $f\in\fsm$ being given by the formula
\[
(fT)(W):=T(fW),\qquad\forall W\in\vsm.
\]
We shall think at the space $\vsm'$ as a kind of `space of vector-valued distributions' on $\X$: our differentiation operators for objects with low regularity will take value in $\vsm'$.

\bigskip

Notice  that the compatibility with the metric of the covariant derivative and the very definition of divergence yield
\begin{equation}
\label{eq:perdistrcov}
\int\la\nabla_XY,W\ra\,\d\mm=\int- \la\nabla_XW,Y\ra- \la Y,W\ra\,\div(X)\,\d\mm\qquad\forall X,Y,W\in \vsm.
\end{equation}
We therefore propose the following definition:
\begin{definition}[Distributional covariant derivative]
Let $X,Y\in L^2(T\X)$ be with $X\in D(\div)$ and so that at least one of $|X|$ and $|Y|$ is in $L^\infty(\X)$. Then $\bnabla_XY\in\vsm'$ is defined as
\begin{equation}
\label{eq:defbnabla}
\bnabla_XY(W):=\int -\la\nabla_XW,Y\ra-\la Y,W\ra\div(X)\,\d\mm\qquad\forall W\in\vsm.
\end{equation}
\end{definition}
Notice that it holds
\[
|\la \nabla_XW,Y\ra-\la Y,W\ra\div(X)|\leq |\nabla W|_{\sf HS}|X||Y|+|Y||W||\div X|,
\]
and recall that for $W\in\vsm$ it holds $|W|\in L^\infty(\X)$ and $|\nabla W|_{\sf HS}\in L^2(\X)$, while $X\in  D(\div)$ means that $\div (X)\in L^2(\X)$. Thus the integral in \eqref{eq:defbnabla} is well defined and the definition of $\bnabla_XY$ is well posed.

The identity \eqref{eq:perdistrcov} grants consistency, i.e.: if  $Y\in W^{1,2}_C(T\X)$ and $X\in D(\div)$ then
\begin{equation}
\label{eq:conscov}
\bnabla_XY(W)=\int\la\nabla_XY,W\ra\,\d\mm,\qquad\forall W\in\vsm.
\end{equation}
Having a notion of distributional covariant derivative leads to the one of distributional Lie bracket:
\begin{definition}[Distributional Lie bracket]
Let $X,Y\in L^2(T\X)$ be with $X,Y\in D(\div)$ and so that at least one of $|X|$ and $|Y|$ is in $L^\infty(\X)$. Then $\distrlie XY\in\vsm'$ is defined as
\[
\distrlie XY:=\bnabla_XY-\bnabla_YX.
\]
\end{definition}
The assumptions we made on $X,Y$ grant that both $\bnabla_XY$ and $\bnabla_YX$ are well defined. Hence so is the case for $\distrlie XY$.
Also, the identity \eqref{eq:conscov} grants that if $X,Y\in W^{1,2}_C(T\X)\cap D(\div)$, then
\[
\distrlie XY(W)=\int \la[X,Y],W\ra\,\d\mm\qquad\forall W\in \vsm.
\]
One of the basic properties of the Lie bracket of smooth vector fields on smooth manifolds is the Jacobi identity, in which two consecutive applications of the brackets occur. One might certainly wonder whether the same holds on $\RCD$ spaces, but in such setting we don't have vector fields regular enough to make twice the Lie bracket operation: the notion of distributional Lie bracket helps in this direction. We start with the following lemma:
\begin{lemma}
For every $X,Y\in\vsm$ and $h:\X\to\R$ Lipschitz and bounded we have $h[X,Y]\in D(\div)$ and
\begin{equation}
\label{eq:divle}
\div(h[X,Y])=\div(X\div(hY)-Y\div(hX)).
\end{equation}
In particular, $[X,Y]\in D(\div)$.
\end{lemma}
\begin{proof}
Let $f\in\fsm$ and notice that
\[
\begin{split}
\int f \div(X\div(hY)-Y\div(hX))\,\d\mm&=\int -\la \nabla f,X\ra\div(hY)+\la \nabla f,Y\ra\div(hX)\,\d\mm\\
&=\int h\la \nabla f,\nabla_YX-\nabla_XY\ra\,\d\mm=\int\la\nabla f,h[X,Y]\ra\,\d\mm.
\end{split}
\]
Thus the first claim follows from the density of $\fsm$ in $W^{1,2}(\X)$ and the very definition of $D(\div)$ and divergence. Then the last claim comes by picking $h\equiv 1$.
\end{proof}
A direct consequence of such lemma is that for $X,Y,Z\in\vsm$ the object $\distrlie{[X,Y]}Z\in\vsm'$ is always well defined. We then have:
\begin{proposition}[Jacobi identity]
For every $X,Y,Z\in\vsm$ we have
\[
\distrlie{[X,Y]}Z+\distrlie{[Y,Z]}X+\distrlie{[Z,X]}Y=0.
\]
\end{proposition}
\begin{proof} 
We need to prove that for a generic $W\in\vsm$ we have
\[
\distrlie{[X,Y]}Z(W)+\distrlie{[Y,Z]}X(W)+\distrlie{[Z,X]}Y(W)=0,
\]
and by linearity we can assume that $W=g\nabla f$ for generic $f,g\in\fsm$. 

We now claim that for every $f,g\in\fsm$ it holds
\begin{equation}
\label{eq:lief}
\distrlie XY(g\nabla f)=\int-X(g)Y(f)-gY(f)\div(X)+Y(g)X(f)+gX(f)\div(Y)\,\d\mm,
\end{equation}
indeed we have
\[
\begin{split}
\bnabla_XY(g\nabla f)&=\int -\la\nabla_X(g\nabla f),Y\ra-gY(f)\div(X)\,\d\mm\\
&=\int-g\H f(X,Y)-X(g)Y(f)-gY(f)\div(X)\,\d\mm,
\end{split}
\]
so that the claim follows subtracting the analogous identity for $\bnabla_YX(g\nabla f)$ using the symmetry of the Hessian.

Replacing $X$ with $[X,Y]$ and $Y$ by $Z$ in \eqref{eq:lief} and using \eqref{eq:divle} with $h\equiv 1$ we obtain 
\[
\begin{split}
\distrlie{[X,Y]}Z(g\nabla f)=\int& -Z(f)\big(XY(g)-YX(g)\big)-\div(X)\big(Y(g)Z(f)+gYZ(f)\big)\\
&+\div(Y)\big(X(g)Z(f)+gXZ(f)\big)+Z(g)\big(XY(f)-YX(f)\big)\\
&+g\div(Z)(XYf-YXf)\,\d\mm.
\end{split}
\]
Adding up the terms obtained by ciclic permutation of $X,Y,Z$  and using the trivial identity
\[
\begin{split}
\int X(f)Y(g)\div(Z)\,\d\mm=-\int Z(X(f))Y(g)+X(f)Z(Y(g))\,\d\mm,
\end{split}
\]
and the `permuted' ones, we get the conclusion.
\end{proof}
We are now ready to give the main definition of this note:
\begin{definition}[Distributional curvature tensor]
For $X,Y,Z\in\vsm$ we define $\ru XYZ\in\vsm'$ as
\[
\ru XYZ:=\bnabla_X(\nabla_YZ)-\bnabla_Y(\nabla_XZ)-\bnabla_{[X,Y]}Z.
\]
For $X,Y,Z,W\in\vsm$ we also define $\rd XYZW\in\fsm'$ as
\[
\rd XYZW(f):=\big(\ru XYZ\big)(fW).
\]
\end{definition}
We conclude pointing out that it is a purely algebraic consequence of the definition, and of the calculus tools developed so far,  that the curvature has the same symmetries it has in the smooth setting (we shall follow the arguments of \cite{Petersen16}):

\begin{proposition}[Symmetries of the curvature] For any $X,Y,Z,W\in\vsm$ and $f\in\fsm$ it holds:
\begin{subequations}
\begin{align}
\label{eq:r1}
\rd XYZW&=-\rd YXZW=\rd ZWXY,\\
\label{eq:r2}
\ru XYZ&+\ru YZX+\ru ZXY=0,\\
\label{eq:r3}
f\rd XYZW&=\rd{fX}{Y}{Z}{W}=\rd{X}{fY}{Z}{W}=\rd{X}{Y}{fZ}{W}=\rd{X}{Y}{Z}{fW}.
\end{align}
\end{subequations}
\end{proposition}
\begin{proof} The first in \eqref{eq:r1} is equivalent to the identity $\ru XYZ=-\ru YXZ$ which in turn is a direct consequence of the definition. 

The equality \eqref{eq:r2} follows from the Jacobi identity for the Lie bracket. Indeed, if for a given trilinear map $T:[\vsm]^3\to\vsm'$ we put  $\mathfrak ST(X,Y,Z):=T(X,Y,Z)+T(Y,Z,X)+T(Z,X,Y)$, we have
\[
\begin{split}
\mathfrak S \ru XYZ&=\mathfrak S\bnabla_X\nabla_YZ-\mathfrak S\bnabla_Y\nabla_XZ-\mathfrak S\bnabla_{[X,Y]}Z\\
&=\mathfrak S\bnabla_Z\nabla_XY-\mathfrak S\bnabla_Z\nabla_YX-\mathfrak S\bnabla_{[X,Y]}Z\\
&=\mathfrak S\bnabla_Z(\nabla_XY-\nabla_YX)-\mathfrak S\bnabla_{[X,Y]}Z\\
&=\mathfrak S\big(\bnabla_Z[X,Y]-\bnabla_{[X,Y]}Z\big)\\
&=\mathfrak S\distrlie Z{[X,Y]}=0.
\end{split}
\]
We now claim that
\begin{equation}
\label{eq:zw}
\rd XYZW=-\rd XYWZ.
\end{equation}
To prove this let $f\in\fsm$ and notice that
\[
\begin{split}
\rd XYZW(f)&=\big(\ru XYZ\big)(fW)\\
&=\int-\la\nabla_X(fW),\nabla_YZ\ra-f\la\nabla_YZ,W\ra\div(X)\\
&\qquad+\la\nabla_Y(fW),\nabla_XZ\ra+f\la\nabla_XZ,W\ra\div(Y)-f\la\nabla_{[X,Y]}Z,W\ra\,\d\mm.
\end{split}
\]
Add to this the equality obtained exchanging $Z$ and $W$ and observe that

\[
\begin{split}
&-\la\nabla_X(fW),\nabla_YZ\ra
-f\la\nabla_YZ,W\ra\div(X)
+\la\nabla_Y(fW),\nabla_XZ\ra
+f\la\nabla_XZ,W\ra\div(Y)\\
&-\la\nabla_X(fZ),\nabla_YW\ra
-f\la\nabla_YW,Z\ra\div(X)
+\la\nabla_Y(fZ),\nabla_XW\ra
+f\la\nabla_XW,X\ra\div(Y)\\
& \quad=-X(f)\Big(\la\nabla_YZ,W\ra+\la\nabla_YW,Z\ra\Big)-f\div(X)\Big(\la\nabla_YZ,W\ra+\la\nabla_YW,Z\ra\Big)\\
&\quad\qquad+Y(f)\Big(\la\nabla_XW,Z\ra+\la\nabla_XZ,W\ra\Big)+f\div(Y)\Big(\la\nabla_XZ,W\ra+\la\nabla_XW,Z\ra\Big)\\
&\quad=-Y(\la Z,W\ra)\div(fX) +X(\la Z,W\ra)\div(fY) ,
\end{split}
\]
and that
\[
\begin{split}
-f\la\nabla_{[X,Y]}Z,W\ra-f\la\nabla_{[X,Y]}W,Z\ra&=-f[X,Y](\la Z,W\ra).
\end{split}
\]
Therefore we have, after an integration by parts, that
\[
\begin{split}
(\rd XYZW&+\rd XYWZ)(f)\\
&=\int \la Z,W\ra\Big(\div\big( Y\div(fX)-X\div(fY)\big)+\div(f[X,Y])\Big)\,\d\mm,
\end{split}
\]
and this latter expression vanishes due to \eqref{eq:divle}. This proves our claim \eqref{eq:zw}.

The equality between the first and last term in \eqref{eq:r1} now follows from the first in \eqref{eq:r1}, \eqref{eq:zw} and \eqref{eq:r2}, indeed:
\[
\begin{split}
\rd XYZW&=-\rd ZXYW-\rd YZXW\\
&=\rd ZXWY+\rd YZWX\\
&=-\rd XWZY-\rd WZXY-\rd ZWYX-\rd  WYZX\\
&=2\rd ZWXY+\rd XWYZ+\rd WYXZ\\
&=2\rd ZWXY-\rd YXWZ\\
&=2\rd ZWXY-\rd XYZW,
\end{split}
\]
i.e.\ $2\rd XYZW=2\rd ZWXY$. 

It remains to prove \eqref{eq:r3}. The chain of equalities
\[
(f\rd XYZW)(g)=\rd XYZW(fg)=\big(\ru XYZ\big)(fgW)=\rd XYZ{fW}(g),
\]
valid for any $f,g\in\fsm$ shows that the first and last term in \eqref{eq:r3} coincide. The equality with the others then follows from \eqref{eq:r1}. 
\end{proof}

\begin{remark}{\rm We should not expect  the trace of the sectional curvature to be equal to the Ricci curvature in any sense: this is not the case not even on weighted Riemannian manifolds, where the correct notion of Ricci tensor is the Bakry-\'Emery one. This should be compared to the fact that the trace of the Hessian of a function is not the Laplacian, in general. 

We believe that the Ricci curvature tensor coincides with the trace of the sectional curvature provided the Laplacian is the trace of the Hessian, but the verification of this fact is outside the scope of this note. We just remark that finite dimensional $\RCD$ spaces for which the latter condition holds are called \emph{weakly non-collapsed $\RCD$ spaces} and that they are conjectured to be \emph{non-collapsed $\RCD$ spaces} (see \cite[Remark 1.13]{GDP17}).
}\fr
\end{remark}
\begin{remark}\label{rem:distr}{\rm The use of the terminology `distributional' that we made here is quite an abuse. Indeed, not only we certainly don't have $C^\infty$ functions in this setting but, most importantly, we didn't put any topology on the spaces $\test\X,\vsm$, so that their dual have been considered only in the algebraic sense.

Still, we chose to stick to the use of `distributional' because in our opinion it gives the idea of what is happening: by throwing derivatives on the appropriate test object we can give a meaning to the Riemann tensor. In any case, it is not hard to put appropriate norms on both $\test\X$ and $\vsm$ so that all the operators we considered are continuous.
}\fr
\end{remark}
\def\cprime{$'$} \def\cprime{$'$}

\end{document}